\documentclass[12pt]{amsart}

\usepackage{etex}
\usepackage{amsmath, amssymb}
\usepackage[frame,cmtip,arrow,matrix,line,graph,curve]{xy}
\usepackage{graphpap, color, paralist, pstricks}
\usepackage[mathscr]{eucal}
\usepackage[pdftex]{graphicx}
\usepackage[pdftex,colorlinks,backref=page,citecolor=blue]{hyperref}
\usepackage{tikz}
\usepackage{kotex}
\usepackage{array}
\newcolumntype{P}[1]{>{\centering\arraybackslash}p{#1}}
\newcolumntype{M}[1]{>{\centering\arraybackslash}m{#1}}

\newtheorem{theorem}{Theorem}[section]
\newtheorem{proposition}[theorem]{Proposition}

\newtheorem{lemma}[theorem]{Lemma}

\newtheorem{conjecture}[theorem]{Conjecture}

\theoremstyle{definition}
\newtheorem{definition}[theorem]{Definition}
\newtheorem{example}[theorem]{Example}

\newtheorem{remark}[theorem]{Remark}

\newcommand{\PP}{\mathbb{P}}
\newcommand{\QQ}{\mathbb{Q}}
\newcommand{\CC}{\mathbb{C}}

\newcommand{\ZZ}{\mathbb{Z}}

\newcommand{\cO}{\mathcal{O} }

\newcommand{\cC}{\mathcal{C} }
\newcommand{\cE}{\mathcal{E} }
\newcommand{\cF}{\mathcal{F} }

\newcommand{\cH}{\mathcal{H} }

\newcommand{\cK}{\mathcal{K} }

\newcommand{\cN}{\mathcal{N} }

\newcommand{\cQ}{\mathcal{Q} }
\newcommand{\cS}{\mathcal{S} }
\newcommand{\cT}{\mathcal{T} }
\newcommand{\cU}{\mathcal{U} }

\newcommand{\rH}{\mathrm{H} }
\newcommand{\rM}{\mathrm{M} }

\newcommand{\rR}{\mathrm{R} }

\newcommand\bG{\mathbf{G}}
\newcommand\bH{\mathbf{H}}

\newcommand\bM{\mathbf{M}}
\newcommand\bN{\mathbf{N}}

\newcommand\fM{\mathfrak{M}}
\newcommand\dual{^{\vee}}

\newcommand\virt{^{\mathrm{vir}} }
\newcommand\vphi{\varphi}

\newcommand\hra{\hookrightarrow}

\newcommand\wtil{\widetilde}

\newcommand{\GW}{\mathrm{GW}}

\newcommand{\DT}{\mathrm{DT}}

\newcommand\Hom{\mathrm{Hom} }
\newcommand\Ext{\mathrm{Ext} }

\newcommand{\Gr}{\mathrm{Gr} }

\newcommand\lr{\rightarrow}

\newcommand{\ses}[3]{0\lr{#1}\lr{#2}\lr{#3}\lr 0}

\newcommand\rk{\mathrm{rk}}

\hypersetup{%
pdftitle={DT-GW correspondence on local Fano 3-folds},%
pdfauthor={Kiryong Chung},%
pdfkeywords={moduli of sheaves, elementary modification, moduli space},%
citecolor=blue,%
linkcolor=blue,%
}

\begin{document}

\title[GW-DT-GV correspondences on local CY 4-folds]{Correspondence of Donaldson-Thomas and Gopakumar-Vafa invariants on local Calabi-Yau 4-folds over $V_5$ and $V_{22}$}

\author{Kiryong Chung}
\address{Department of Mathematics Education, Kyungpook National University, 80 Daehakro, Bukgu, Daegu 41566, Korea}
\email{krchung@knu.ac.kr}

\author{Sanghyeon Lee}
\address{School of Mathematics, Korea Institute for Advanced Study, 85 Hoegiro, Dongdaemun-gu, Seoul 02455, Korea}
\email{sanghyeon@kias.re.kr}

\author{Joonyeong Won}
\address{Center for Mathematical Challenges, Korea Institute for Advanced Study, 85 Hoegiro, Dongdaemun-gu, Seoul 02455, Korea}
\email{leonwon@kias.re.kr}


\begin{abstract}
 We compute Gromov-Witten (GW) and Donaldson-Thomas (DT) invariants (and also descendant invariants) for local CY $4$-folds over Fano $3$-folds, $V_5$ and $V_{22}$ up to degree $3$. We use torus localization for GW invariants computation, and use classical results for Hilbert schemes on $V_5$ and $V_{22}$ for DT invariants computation. From these computations, one can check correspondence between DT and Gopakumar-Vafa (GV) invariants conjectured by Cao-Maulik-Toda in genus $0$. Also we can compute genus $1$ GV invariants via the conjecture of Cao-Toda, which turned out to be $0$. These fit into the fact that there are no smooth elliptic curves in $V_5$ and $V_{22}$ up to degree $3$. 
\end{abstract}

\maketitle
\section{Introduction}

\subsection{Motivation}
Recently, Donaldson-Thomas (DT) invariants for Calabi-Yau 4-folds and the correspondence among related invariants, like stable pair (PT) and Gopakumar-Vafa invariants (GV or BPS) have been actively studied in \cite{CL14, CMT21, CMT18, CK20, CK18, CKM21}.

Virtual cycles of $\mathrm{DT_4}$ moduli space has been developed by Borisov-Joyce \cite{BJ17} and by Oh-Thomas \cite{OT20}. For a CY 4-fold $X$ and a curve class $\beta \in \rH_2(X,\ZZ)$, we denote $M_{\beta}(X)$ the moduli space of 1-dimensional stable sheaves $F$ on $X$, such that $[F]=\beta$ and $\chi(F)=1$. Then $\DT_4$ invariants are defined by integrations of cohomology insertions over virtual cycles $[M_{\beta}(X)]\virt$. Note that the virtual cycle is defined via some suitable choice of an orientation on the moduli space $M_{\beta}(X)$.

Insertions are defined by the following way. Consider the product space $M \times X$ and a (normalized)universal sheaf $\cF$ over it. For a cohomology element $\gamma \in \rH^*(X, \ZZ)$, let
\begin{align*}
\tau_i(\gamma) := (\pi_M)_*( \pi_X^*\gamma \cup \mathrm{ch}_{\dim(X)+i-1}(\cF) ) \in \rH^{*+2(i-1)}(M, \ZZ)
\end{align*}
where $M$ is an abbreviation of $M_{\beta}(X)$.
$\mathrm{DT_4}$ invariants and descendent invariants (with a single insertion) are defined by
\[
\langle \tau_i(\gamma) \rangle_{\beta} := \int_{[M_{\beta}(X)]\virt} \tau_i (\gamma)
\]
via some suitable choice of orientation on the moduli space $M_{\beta}$. $\langle \tau_0(\gamma) \rangle_{\beta}$ is called \emph{$\DT_4$ invariants}, and we denote it by $\DT_4(X)(\beta|\gamma)$.
When $i \geq 1$, we call $\tau_i(\gamma)$ \emph{descendant insertions}, and integrations of descendant insertions over the virtual class are called \emph{descendant invariants}.

On the other hand, we can consider $\mathrm{DT_3}$ invariants on Fano 3-folds. Note that $\DT_3$ moduli space for a Fano 3-fold is isomorphic to the $\DT_4$ moduli space of the corresponding local CY 4-fold \cite[Theorem 6.5]{CL14}. Moreover, $\mathrm{DT_3}$ invariants for a Fano 3-fold is equal to the $\DT_4$ invariant of the corresponding local CY 4-fold when we choose some suitable orientation in the $\DT_4$ moduli space \cite{Cao19a, Cao19b}. 

Now we review some famous conjectures between $\mathrm{DT_4}$ and GV invariants. Let us denote $n_{0,\beta}(\gamma)$ (resp. $n_{1,\beta}$) be the genus 0 (resp. genus 1) Gopakumar-Vafa invariants on a CY 4-fold $X$ defined in \cite{KP08}.
\begin{conjecture}\cite[Conjecture 0.2]{CMT18}\label{conj:genus0}
Via some suitable choice of orientation on the moduli space, we have
\begin{equation}\label{gwdtcor}
n_{0,\beta}(\gamma) = \mathrm{DT}_4(\beta|\gamma), \;\mathrm{GW}(\gamma)^{X}_{0,\beta} = \sum_{k|\beta} \frac{1}{k^2} \mathrm{DT}_4(X)(\beta/k | \gamma)
\end{equation}
for all $\gamma \in \rH^4(X, \ZZ)$.
\end{conjecture}
Note that two equations in \eqref{gwdtcor} are equivalent via the definition of $n_{0,\beta}(\gamma)$ in \cite{KP08}. In \cite{CMT18}, 
Conjecture \ref{conj:genus0} is proven for some cases: CY 4-folds with a elliptic structure and local curves/surfaces.
Also, in \cite{CT21}, the following conjecture is proposed which relates $\DT_4$ descendent invariants and genus one GV invariants. Let $m_{\beta_1, \beta_2}$ be the meeting invariants defined in Section 0.3 in \cite{KP08} (for a detail, see Section \ref{meeting}).
\begin{conjecture}\cite[Conjecture 0.2]{CT21}\label{conj:genus1}
Via some suitable choice of orientation on the moduli space, we have
\begin{align*}
 \langle \tau_1(\gamma) \rangle_{\beta} & =\frac{n_{0,\beta}(\gamma^2)}{2(\gamma\cdot \beta)}-\sum_{\beta_1+\beta_2=\beta} \frac{(\gamma\cdot \beta_1)(\gamma\cdot \beta_2)}{4(\gamma\cdot \beta)} m_{\beta_1,\beta_2} \\ 
& -\sum_{k\geq1, k|\beta} \frac{(\gamma\cdot \beta)}{k}n_{1,\beta/k}
\end{align*}
for all $\gamma \in \rH^2(X,\ZZ)$.
\end{conjecture}
Here, the number $m_{\beta_1,\beta_2}$ is a virtual count of degree $\beta_1$ curves in $X$ which meets with degree $\beta_2$, called \emph{meeting invariants}. We will compute these numbers in Section \ref{meeting}.

In \cite{CT21}, Conjecture \ref{conj:genus1} is proved in several cases:
CY 4-folds with an elliptic fibration structure and local Fano $3$-fold over
 ${\PP^3}$. In this paper, as a continuation of the latter case, our study focus on the Fano $3$-folds: $V_5$ and $V_{22}$. See Section \ref{casev5} and \ref{casev22} for definitions of $V_5$ and $V_{22}$. These Fano $3$-folds arise as a minimal compactification of the complex $3$-space $\CC^{3}$. For detailed descriptions, see the beginning parts of Section \ref{section:DTcomp}.
 
 \subsection{Main results}
In this paper, we compute some low degree $\mathrm{DT_3}$ invariants for Fano 3-folds $V_5$ and $V_{22}$. 
For the cases $\beta = d\cdot[\text{line}]$, $1\leq d \leq 3$ where $[\text{line}]$ is a class of $[\PP^1]$ in their projective embedding via the very ample generators, we review some classical results on description of $\DT_3$ moduli space of $V_5$ and $V_{22}$ and universal sheaves on them. Using this, we will compute $\DT_3$ invariants of these Fano 3-folds(equivalently $\DT_4$ invariants of corresponding local CY 4-fold) in Section \ref{section:DTcomp}.

On the other hand, in Section \ref{section:grassGW} we compute twisted Gromov-Witten (GW) invariants on $V_5$ and $V_{22}$ for the cases $\beta = d\cdot[\text{line}]$, $1\leq d \leq 3$, using quantum Lefschetz property \cite{KKP03} and torus localization \cite{GP99}. We will introduce some recipe for computing genus zero GW invariants on Grassmannian varieties, which is a direct analogue of the formula in \cite{GP99, MS} for genus zero GW invariants on projective spaces. From these calculations, we obtain the following main results of the paper.
\begin{theorem}[=Theorem \ref{main1}]
Using the choice of orientation as in \cite[(0.7)]{CT21} on $\DT_4$ moduli spaces on local CY 4-folds, the Conjecture \ref{conj:genus0} holds for $V_5$ and $V_{22}$, $\beta = d\cdot [\text{line}]$, $1 \leq d \leq 3$.
\end{theorem}
\begin{theorem}[=Theorem \ref{pfmeeting}]
Using the choice of orientation as in \cite[(0.7)]{CT21} on $\DT_4$ moduli spaces on local CY 4-folds, the Conjecture \ref{conj:genus1} holds for $V_5$ and $V_{22}$, $\beta = d\cdot [\text{line}]$, $1 \leq d \leq 3$ if and only if genus $1$ GV invariants $n_{1,\beta}\equiv 0$.
\end{theorem}
From the fact that the degree of the defining equation of $V_5$ and $V_{22}$ is $\leq 2$ and the varieties does not contain plane, one can check that the smooth elliptic curves of the degree $\leq 3$ in $V_5$ and $V_{22}$ does not exist. Hence the vanishing of GV invariants $n_{1,\beta}$ is desirable.

We have posted all source codes we have used
during the computation of GW invariants and their outputs at the second author's website:
\begin{center}
\href{https://sites.google.com/view/sanghyeon-lee/reference?authuser=0}{https://sites.google.com/view/sanghyeon-lee/reference?authuser=0}
\end{center}

\subsection*{Acknowledgements}
The authors gratefully acknowledge the many helpful suggestions of Yalong Cao and Young-Hoon Kiem during the preparation of the paper. The second named author thanks to Jeongseok Oh and Hyeonjun Park for much advice on Donaldson-Thomas invariants.
\section{Twisted GW invariants on Grassmannians}\label{section:grassGW}
Let $Y$ be a smooth Fano $3$-fold embedded in the Grassmannian variety $G = \Gr(r,d)$, which is a zero section of a bundle over $G$. In this subsection, we present an algorithm of the computation of GW invariants of $|K_Y|$, a total space of a canonical line bundle of $Y$. For any smooth projective variety $X$, we denote the moduli space of stable maps, with  $k$-marked points with degree $\beta \in \rH^2(X,\ZZ)$ by $M_{0,k}(X,\beta)$. When $X = G$, we will denote $M_{0,k}(G,d)$ the stable map space with degree $d\cdot[\text{line}]$. We will usually abbreviate $M_{0,k}(X,\beta)$ by $M(X)$ in the following. In this section, we will introduce classical methods to compute GW invariants, quantum Lefschetz principle and torus localization when the target variety is $|K_Y|$. If you are not interested in this part, you may skip the details and just see computational results in Section \ref{GWcompforFano}.

Consider the forgetful map $p : M=M_{0,k}(G,d) \to \fM_{0,k}$ where $\fM_{0,k}$ is the moduli space of prestable genus $0$ curves with $k$ marked points. We usually will abbreviate $\fM_{0,k}$ by $\fM$.
There is a usual (relative) perfect obstruction theory \cite{Beh96} $E_{M(X)/\fM} \to L_{M(X)/\fM}$ of $M(X) \to \fM$, where
\begin{align}\label{GWpot}
E_{M/\fM} := [R\pi_* f^* T_G]\dual \in D^b(M(X)).
\end{align}
Also, as in \cite{BF97} the virtual cycle correspond to this perfect obstruction theory is given by:
\begin{align}\label{virtcycle}
[M(X)]\virt := 0^!_{h^1/h^0(E_{M/\fM})}[\cC_{M(X)/\fM}] \in A_{\mathrm{vdim}}(M(X))
\end{align}
where $h^1/h^0(E_{M/\fM})$ is a vector bundle stack correspond to $E_{M/\fM}$, $\cC_{M(X)/\fM}$ is the (relative) intrinsic normal cone defined in \cite{BF97}, and $\mathrm{vdim}$ is given by
\[
\mathrm{vdim} = (1 - g)(\dim(X) - 3) + k - \int_{\beta} c_1(K_X)
\]
which is called virtual dimension.

\subsection{Quantum Lefschetz principle}\label{section:QLef}
We will briefly review some aspect of famous quantum Lefschetz principle \cite{KKP03} in this section. We consider a negative vector bundle $E$ on $G$, so that $\rH^0(G,f^*E)=0$ for any non-constant morphism $f : \PP^1 \to G$.
Let $|E|$ be the total space of $E$ and $p : |E| \to G$ be the projection. Consider a stable map $f : C \to |E|$. Then $p\circ f : C \to G$ is a stable map and $f$ induces an element of $\rH^0(C, (p\circ f)^*E)$. But since $E$ is a direct sum of negative degree line bundles, $\rH^0(C, (p\circ f)^*E) = 0$. 
Hence we have a natural isomorphism of moduli spaces of stable maps:
\[
M_{0,k}(|E|,d) \cong M_{0,k}(G,d).
\]
Note that $M_{0,k}(G,d)$ is well known to be smooth. Also, we have the short exact sequence: 
\[
0 \to p^*T_G \to T_{|E|} \to p^*E \to 0. 
\]
From the condition that $E$ is negative and the fact that $\rR^1\pi_* (f^* p^* T_G) = 0$, we obtain the following by taking the higher direct image functor $R \pi_* f^*$ to the above short exact sequence.
\[
E_{M(|E|)/\fM}\dual = \rR\pi_* T_{|E|} = \left[ \pi_*(f^*p^* T_G) \stackrel{0}{\longrightarrow} R^1\pi_*(f^*p^*E)  \right]
\]
Note that $\pi_*(f^*p^*T_G) \cong T_{M(G)}$. We can easily check that the intrinsic normal cone $\cC_{M(G)/\fM}$ is $[M(G) / T_{M(G)/\fM}]$ and $h^1/h^0(E\dual) = [R^1\pi_*(f^*p^*E) / T_{M(G)/\fM}]$.
From the definition of the virtual cycle in \eqref{virtcycle}, we have
\begin{align*}
[M(G)]\virt & = 0^!_{[R^1\pi_*(f^*p^*E) / T_{M(G)/\fM}]}[M(G) / T_{M(G)/\fM}] \\
& = 0^!_{R^1\pi_*(f^*p^*E)}[M(G)] \\
& = e( \rR^1\pi_*(f^* p^*E) ) \cap [M(G)]
\end{align*}
where the second identity comes from properties of Gysin pull-backs via bundle stacks in \cite{Kre99}.
We call this phenomenon \textbf{B-twist}. Note that this phenomenon also arises when we replace $G$ by any projective variety.
Using this, we define a \emph{twisted GW invariants} of a Fano variety $X$ by:
\begin{definition}[Twisted GW invariant]
The twisted Gromov-Witten invariant a Fano variety $X$ is defined by the integration
\[\GW^{\text{twist}}_{0,\beta}(X)(\gamma):=\int_{[M_{0,1}(X,\beta)]^{\mathrm{vir}}} e(\rR^1\pi_*f^*K_X)\cup \mathrm{ev}^*(\gamma) \in \QQ\]
where $\mathrm{ev}: M_{0,1}(X,\beta) \lr X$ be the evaluation map.
\end{definition}
Note that twisted GW invariants are usually considered as a definition of GW invariants of the total space $|K_X|$, because we can not define it directly since $|K_X|$ is not compact.

Next we consider a complete intersection in $G=G(r,n)$. Let $E$ be the vector bundle on $G$ and let $Y \subset G$ be a zero section of the generic section $s : \cO_G \to E$. Assume that $E$ is convex, so that $\rH^1(G, \vphi^*E) = 0 $ for any non-constant morphism $\vphi : \PP^1 \to G$. For example, the direct sum $\bigoplus_i \cO(a_i)$, $a_i > 0$ of line bundles on $G$ is a convex vector bundle.
From the convexity of $E$, $\pi_* f^*E$ is locally free, where $\pi : \cC \to M$ is the universal curve and $f : \cC \to G$ is the universal morphism. Let $M(Y)$ denote the stable map space $M_{0,k}(Y,d)$. Then we have
\[
E_{M(Y)/\fM}\dual = \rR\pi_* f^* T_Y = [\pi_*f^* T_G \to \pi_* f^* E] = [T_{M} \to \pi_* f^* E].
\]

Consider the section $s : \cO_G \to E$ and the induced section $\wtil{s} : M \to  \pi_*f^*E $ defined by $\wtil{s}([(C,f)]) = (C, f^*s) \in H^0(C, f^*E )$. Then we have $Z(\wtil{s}) = M(Y) \subset M$.
We can easily check that $\cC_{M/\fM} = [C_{M(Y)/M} / T_{M/\fM}]$ and $h^1/h^0(E\dual) = [\pi_* f^* E / T_{M/\fM}]$.
Then by the definition of the virtual cycle, we have
\[
[M(Y)]\virt = 0^!_{[\pi_* f^* E / T_M]} [C_{M(Y)/M} / T_{M/\fM}] = 0^!_{\pi_* f^* E} [C_{M(Y)/M}] \in A_*(M(Y)).
\]
Also we have
\[
\iota_*[M(Y)]\virt = e( \pi_* f^*E ) \cap [M] \in A_*(M)
\]
where $\iota : M(Y) \hra M$ is the inclusion.
See \cite{KKP03} for the proof of the general case. We call this phenomenon \textbf{A-twist}. 
\subsection{Torus localization}\label{section:Torloc}
We will briefly review some aspect of torus localization \cite{GP99} in this section. Again we consider the Grassmannian variety $G = \Gr(r,d)$ and the stable map space $M = M_{0,k}(G,d)$.
We give a $\CC^*$-action on $\CC^n$ with weights $-\alpha_1,\dots, -\alpha_n$, which induces the $\CC^*$-action on $\Gr(r,n)$ and $M$. 
Let $M^F \subset M$ be the fixed locus of the action and let $M^F =  \bigcup_i M^F_i$ be the irreducible decomposition.
Note that $E_{M/\fM}\dual$ has an induced $\CC^*$-action and we have a decomposition
\[
E_{M/\fM}\dual|_{M^F_i} \cong N^{\mathrm{fix}} \oplus N\virt
\]
where $N^{\mathrm{fix}}$ has weight $0$ under the $\CC^*$-action and $N\virt$ is a direct sum of vector bundles with non-zero weights.

Let $(A^T)_*(M)$ (resp. $(A^T)^*(M)$) be the equivariant Chow group (resp. equivariant Chow cohomology group) of  $M$ and  $e^T (E)$ be the equivariant Euler class of a locally free sheaf $E$. If $E_M\dual$ has locally free resolution $[E_0 \to E_1]$, we define the (equivariant) Euler class by
\[
e^T(N\virt) = \frac{e^T(E_0^m)}{e^T(E_1^m)} \in (A^T)^*(M) \otimes \QQ[t,1/t].
\]
By the virtual localization theorem in \cite{GP99}, we have
\[
[M]\virt = \sum_{i} \frac{[M^F_i]\virt}{e^T(N\virt)} \in A^T_*(M) \cong (A^T)_*(M)\otimes \QQ[t,1/t].
\]
\subsection{Computation of the virtual normal bundle}\label{section:Virtnor}
For $M=M_{0,k}(G,d)$, we can do more specific computation. In a similar manner as in \cite{GP99} and \cite{MS}, which dealt with the case $G = \PP^N$, fixed loci $M^F_i$ are indexed by decorated graphs $\Gamma$. We denote $M^F_i$ by $M_{\Gamma}$ for the corresponding decorated graph $\Gamma$. Note that $M_{\Gamma}$ is smooth and thus $[M_{\Gamma}]\virt = [M_{\Gamma}]$. Also $N\virt$ denotes the virtual normal bundle defined in \cite{BF97}.

For a stable map $[(C,x_1,\dots, x_k,f)] \in M_{\Gamma}$, the fiber of the (K-theoretic) virtual normal bundle is given by the moving part of $\Ext^1(\Omega_C(x_1+\dots,x_k), \cO_C ) - \Ext^0(\Omega_C(x_1+\dots,x_k), \cO_C ) + (\rH^0 - \rH^1)(f^*T_G)$. 
By \cite{GP99} and \cite{MS}, we have
\[
e^T( \Ext^0(\Omega_C(x_1+\dots,x_k), \cO_C ) ) = \prod_{\substack{v \in \mathrm{Vertices} \\ \text{val}(v)=1 } }\omega_{F_v} 
\]
and
\[
e^T(\Ext^1(\Omega_C(x_1+\dots,x_k), \cO_C )) = \prod_{\substack{F \in \mathrm{Flags} \\ \mathrm{valency}(i(F)) \geq 3} }(\omega_F - e_F),
\]
where $\omega_F := \frac{\alpha_{i(F)} - \alpha_{j(F)} }{d_e}$ and $e_F$ is the $\psi$-class correspond to the flag $F$. Also, by \cite{GP99} and \cite{MS}, we have the following by using the normalization sequence of nodal curves.
\begin{align}\label{decomp}
(\rH^0-\rH^1)(f^*T_G) & = \bigoplus_{v \in \mathrm{Vertices}} T_{p_v } G + \bigoplus_{e \in \mathrm{Edges}} \rH^0(C_e, f^*T_G) \\ \nonumber
& - \bigoplus_{F \in \mathrm{Flags}}T_{p_{i(F)}} G - \bigoplus_{v \in \mathrm{Vertices}} \rH^1(C_v, f^*T_G)  
\end{align}
Then the equivariant Euler classes of their moving part are given by the followings. For the first term of \eqref{decomp}, if $p_v = x_{u_1,\dots,u_r} = \langle e_{u_1},\dots,e_{u_r} \rangle \in \Gr(r,n)$, then
\[
e^T(T_{p_{i(v)}} G ) = \prod_{1 \leq j \leq r}\prod_{k \in [n]\setminus \{ u_1,\dots,u_r \} } ( \alpha_{u_j} - \alpha_k ).
\]
For the third term of \eqref{decomp}, if $p_{i(F)} = x_{u_1,\dots,u_r} = \langle e_{u_1},\dots,e_{u_r} \rangle \in \Gr(r,n)$, then we have the same formula as above :
\[
T_{p_{i(F)}} = \prod_{1 \leq j \leq r}\prod_{k \in [n]\setminus \{ u_1,\dots,u_r \} } ( \alpha_{u_j} - \alpha_k ).
\]
For the fourth term of \eqref{decomp}, by \cite{MS}, we have the following. If valency($v$)$=2$ and there is no marking on $v$, then
\[
e^T(\rH^1(C_v, f^*T_G)) = \omega_{F_{v,1}} + \omega_{F_{v,2}}.
\]
Otherwise, we have
\[
e^T(\rH^1(C_v,f^*T_G))=1.
\]
For the second term of \eqref{decomp}, consider the Euler sequence:
\begin{align}\label{Euler}
0 \to S\dual \otimes S \to S\dual \otimes \cO_G^{\oplus n} \to S\dual \otimes Q = T_G \to 0
\end{align}
where $S,Q$ are tautological bundle and universal quotient bundle on $G$. Take pull-back via the map $f : C_e \to G$. Let $e = \{v_1, v_2 \}$, $p_{v_1} = x_{u_1,\dots,u_{r-1}, u_{a}}$ and $p_{v_2} = x_{u_1,\dots,u_{r-1},u_{b}}$. Note that $u_{a} \neq u_{b}$. Then we have 
$f^*(S\dual) \cong \cO(\alpha_{u_1})\oplus...\oplus \cO(\alpha_{u_1}) \oplus \cO(d_e)$
where $\cO(\alpha_i)$ is an equivariant trivial bundle where $(\CC^*)^n$ acts on it with a weight $\alpha_i$. Note that $\rH^0(C_e, \cO(d_e))$ has weights $\frac{c_1 \alpha_{u_a} + c_2\alpha_{u_b}}{d_e}$ for $c_1+c_2 = d_e$. Let $f|_{C_e} =: f_e$. By taking $f_e^*$ and the cohomology in \eqref{Euler}, we have the following exact sequence:
\[
0 \to \rH^0(f_e^* (S\dual \otimes S) ) \to \rH^0( (f_e^* S\dual \otimes \cO_G^{\oplus n} ) ) \to \rH^0(f_e^* T_G) \to \rH^1( (f_e^* S\dual \otimes S ) ) \to 0.
\]
By a direct calculation, we have
\begin{align*}
&e^T( (\rH^1-\rH^0)( f_e^* (S\dual \otimes S) ) ) \\
& = \frac{ (-1)^{d_e-1} }{ \prod_{1 \leq i \leq r-1} (\alpha_{u_a}-\alpha_{u_i})(\alpha_{u_b}-\alpha_i) \prod_{1 \leq i < j \leq r-1} (\alpha_{u_i} - \alpha_{u_j})^2 }.  
\end{align*}
Also we have,
\begin{align*}
e^T( (\rH^1 - \rH^0)(f_e^*S\dual \otimes \cO^{\oplus n}_G) ) & = \prod_{1 \leq j \leq r}\prod_{k \in [n]\setminus \{ u_1,\dots,u_r \} } ( \alpha_{u_j} - \alpha_k ) \\
& \times \prod_{k \in [n]\setminus \{ u_1,\dots,u_r \}} \prod_{\substack{c_1,c_2\neq 0 \\ c_1+c_2=d_e \\ (c_1,k) \neq (0,u_b), \, (d_e,u_a)} } (\frac{c_1 \alpha_{u_a} + c_2\alpha_{u_b}}{d_e} - \alpha_{k} ). 
\end{align*}

\subsection{Computation of GW invariants on Fano $3$-folds}\label{GWcompforFano}

Combining above arguments, we have expression of $e^T(N\virt|_{M_{\Gamma}})$. Next we represent A-twist and B-twist in section \ref{section:QLef} in equivariant cohomology. Let $F_1:=\pi_* f^* \left( \bigoplus_i \cO(a_i) \right)$ and $F_2:=R^1\pi_* f^* \cO(-b)$ where $b$ is the Fano index of $X$, $\pi : \cC \to M(X)$ is the universal curve and $f : \cC \to X$ is the universal morphism. We have
\begin{align}\label{localization+QLef}
\iota_*[M(X)]\virt =  \sum_{\Gamma} \frac{1}{|A_{\Gamma}|} \frac{e^T(F_1|_{M_{\Gamma}})\cup e^T(F_2|_{M_{\Gamma}})}{e^T(N\virt)} [M_{\Gamma}] \in (A^T)_*(M)\otimes \QQ[t,1/t]
\end{align}
where $A_{\Gamma}$ is the order of the automorphism group of a generic element in $M_{\Gamma}$. We have $|A_{\Gamma}| = |\mathrm{Aut}(\Gamma)| \cdot \prod_{e \in \mathrm{Edges}} d_e$ where $\mathrm{Aut}(\Gamma)$ is the automorphism group of the decorated graph $\Gamma$. One can check more detail on the group $A_{\Gamma}$ in \cite{GP99}.
Note that we can find specific expressions of $e^T(F_1|_{M_{\Gamma}})$ and $e^T(F_2|_{M_{\Gamma}})$ using the normalization sequence:
\begin{align}\label{normalseq}
0 \to \cO_C \to \bigoplus_{v \in \mathrm{Verteces}} \cO_{C_v} \oplus \bigoplus_{e \in \mathrm{Edges}} \cO_{C_e}
 \to \bigoplus_{\substack{F \in \mathrm{Flags}} } \CC_{x_F} \to 0.  
\end{align}
Combining arguments in section \ref{section:Torloc}, \ref{section:Virtnor}, \ref{section:QLef} and Hodge integrals computed in \cite{FP00}, we can express the right hand side terms of \eqref{localization+QLef} by formal weights of $\CC^*$-action. Therefore, we can compute genus 0 Gromov-Witten invariants of the total space of the canonical line bundle over the Fano 3-fold, which is a zero section of an equivariant vector bundle over a Grassmannian variety.

In this paper, we consider two cases: (a) $Y = V_5$ and (b) $Y = V_{22}$. In case (a), $F_1 = \pi_* f^* \cO(1)^{\oplus 3}$ and $F_2  = \rR^1\pi_* f^* \cO(-2)$. In case (b), $F_1 = \pi_* f^* (\wedge^2 S\dual )$ and $F_2 =\rR^1\pi_* f^* \cO(-1)$. The actual computation has been done by a computer program. Firstly we make a dataset of all possible decorated graphs $\Gamma$ and their information. Secondly, using this dataset, we make a code computing the right hand side terms in \eqref{localization+QLef} for each localization graph $\Gamma$ and adding up them. As a result, we obtain the following table are twisted GW invariants.
\begin{proposition}[Twisted GW invariants]\label{gwfano}
The twisted GW invariants for $Y=V_5$ and $V_{22}$ are given by the numbers of the following table.
\begin{center}
\begin{tabular}{|l||M{3.2cm}|M{3.2cm}|}
\hline
$d$&$\GW^{\text{twist}}_{0,d}(V_5)(h_2)$& $\GW^{\text{twist}}_{0,d}(V_{22})(h_2)$\\
\hline
\hline
$1$&$-5$&$2 $\\
\hline
$2$&$-36\frac{1}{4}$&$-6\frac{1}{2}$\\
\hline
$3$&$-490\frac{5}{9}$&$28\frac{2}{9}$\\
\hline
\end{tabular}
\end{center}
Here $h_2$ is the generator of $\rH^4(Y,\ZZ) \cong \ZZ$ .
\end{proposition}
\begin{remark}
The degree $4$ twisted GW invariant on $V_5$ is given by $\GW^{\text{twist}}_{0,4}(V_5)(\gamma_2)=-8829\frac{1}{16}$. In principle we can compute GW invariants for higher degrees, but the time taken for the calculation super-exponentially increases. 
\end{remark}
\section{Donaldson-Thomas type invariants}
In this section, we compute DT invariants and descendant invariants for some local Fano 3-folds $|K_{V_5}|$ and $|K_{V_{22}}|$ for degree $1 \leq d\cdot [\textrm{line}] \leq 3$. We will abbreviate $d\cdot [\textrm{line}]$ by $d$.
In these cases, $M_{d}$ naturally isomorphic to moduli space of stable sheaves on $|K_Y|$. 

\begin{definition}[Twisted DT invariant]
Let $\bM_{\beta}=\bM_{\beta}(Y)$ be the moduli space of stable sheaves $F$ on $Y$ with $[F]=\beta\in H_2(Y,\ZZ)$ and $\chi(F)=1$. Let
\[
\tau_0:\rH^4(Y,\ZZ)\lr\rH^2(\rM_{\beta},\ZZ),\;\tau_0(\gamma)=\pi_{M*}(\pi_X^*\gamma\cup \mathrm{ch}_2(\cF))
\]
be the \emph{primary insertion} of $\gamma\in \rH^4(Y,\ZZ)$. Here $\cF$ is the universal sheaf and the maps $\pi_{M_\beta}$, $ \pi_Y$ are the canonical projection maps.
The \emph{twisted} genus zero DT invariant is defined by
\[
\DT_3^{\text{twist}}(Y)(\beta|\gamma):= (-1)^{c_1(Y)\cdot \beta-1}\int_{[\bM_{\beta}(Y)]^{\mathrm{vir}}} \tau_0(\gamma) \in \ZZ.
\]
where $[\bM_{\beta}(Y)]^{\mathrm{vir}}$ is the virtual class defined in \cite[Corollary 3.39]{Tho00}.
\end{definition}
Since $\DT_4$ invariant of $|K_Y|$ is equal to $\DT_3$ invariant of $Y$, it is enough to compute twisted $\DT_3$ invariant.

Note that if the moduli space $\bM_{\beta}(Y)$ is smooth, the virtual cycle is the Poincar\'e dual of the top Chern class of the obstruction bundle.
Combining computation of GW invariants in Section \ref{section:grassGW} and DT invariant computation in this section, we will check Conjecture \ref{conj:genus0} for $1\leq d \leq 3$, which can be rewritten by

\begin{conjecture}\label{gw-dtconj}
For the cohomology class $\gamma\in \rH^4(Y,\ZZ)$,
\[
n_{0,\beta}(\gamma)=\DT_4(|K_Y|)(d |\gamma)
\]
and
\[
\GW_{0,\beta}(Y)^{\text{twist}}(\gamma)= \sum_{k|\beta}\frac{1}{k^2} \cdot \DT_4(|K_Y|)(d/k|\gamma).
\]
\end{conjecture}

On the other hand, Cao and Toda suggest genus one GV-type invariant on CY $4$-fold $X$ by using the descendent insertion in \cite{CT21}. 

\begin{definition}[Descendent insertion]
For an integral class $\gamma\in \rH^{2}(X,\ZZ)$, let us define the \emph{descendent insertion} as
\[
\tau_{1}:\rH^{2}(X,\ZZ)\lr\rH^2(\rM_{\beta},\ZZ),\;\tau_{1}(\gamma)=\pi_{M*}(\pi_X^*\gamma\cup \mathrm{ch}_{4}(\cF_{\text{norm}})),
\]
where $\cF_{\text{norm}}:=\cF\otimes \pi_M^*(\pi_{M*}\det(\cF))^{-1}$ is the normalized universal sheaf of the universal sheaf $\cF\in \text{Coh}(\rM_{\beta}\times X)$
and the maps $\pi_{M}$ and $ \pi_X$ are the canonical projection maps from $M_{\beta} \times X$ into $M_{\beta}$ and $X$ respectively.
\end{definition}
Descendant invariants from the \emph{descendent insertion}s are defined by
\[
\langle \tau_1(\gamma)\rangle_{\beta}:=\int_{[\bM_{\beta}]^{\mathrm{vir}}} \tau_1(\gamma). 
\]
\begin{remark}
 For $\gamma\in \rH^{4-2i}(X,\ZZ)$, the insertion becomes
\[
\tau_{i}(\gamma)=\pi_{M*}(\pi_Y^*\gamma\cup \{\mathrm{ch}(\cF_{\text{norm}})\cdot \text{td}(K_Y)^{-1}\}_{2+i})
\]
by Grothendieck-Riemann-Roch theorem.
\end{remark}
By computing descendant invariants, we can obtain genus 1 GV invariants $n_{1,\beta}$ via Conjecture \ref{conj:genus1}.
\subsection{Computations on Fano $3$-folds}\label{section:DTcomp}
It is very well-known that the following list of smooth Fano $3$-folds have the same Betti numbers of that of $\PP^3$:
\[
\PP^3, \;Q_2\subset \PP^4, \;V_5\subset \PP^6, \;V_{22}\subset \PP^{13}.
\]
In special, the odd cohomology of these varieties vanish. All of varieties are rigid except $V_{22}$. The moduli of $V_{22}$'s is six-dimensional. The first two varieties in the list are homogeneous and the others are not. In this section, we compute the primary and descent invariants for non-homogeneous cases.
\subsubsection{The case $V_5$}\label{casev5}
The Fano threefold $V_5$ is defined by the linear intersection $V_5=\Gr(2,\CC^5)\cap H_1\cap H_2 \cap H_3$ where $H_i$ are the general hyperplane in $\PP(\wedge^2\CC^5)=\PP^9$.
\begin{itemize}
\item $\mathrm{Pic}_{\ZZ}(V_{5})\cong \ZZ\langle H\rangle$, $\mathrm{deg}(V_{5})=5$, $K_{V_{5}}=-2H$.
\item The cohomology ring of $Y$ over $\ZZ$ is isomorphic to
\[
\rH^{*}(Y_{5},\ZZ)\cong\ZZ[h_1,h_2,h_3]/\langle h_1^2-5h_2,h_1^3-5h_3,h_1h_2-h_3,h_2^2\rangle
\]
where $\deg(h_i)=2i, 1\leq i\leq 3$. Moreover, $h_1=c_1(\cO_Y(1))$ and $h_i$ is the Poincar\'e dual of the linear space of dimension $3-i$ for $i=2,3$.
\item Let $\cS_Y$ and $\cQ_Y$ be the restriction of the universal bundles $\cS$ and $\cQ$ on $\Gr(2,5)$. The Chern classes are
\begin{enumerate}
\item $c(\cS_{V_5})=1-h_1+2h_2$,
\item $c(\cQ_Y)=1+h_1+3h_2+h_3$,
\item $c(Y)=1+2h_1+12h_2+4h_3$.
\end{enumerate}
Unless otherwise stated, we omit the subscription $V_5$ in the universal bundles.
\end{itemize}
Let us denote by $\bM_d:=\bM_{\beta}(V_5)$ for $\beta=d [\text{line}]\in \rH_2(V_5,\ZZ)$. For $1\leq d \leq 3$, one can easily see that $\bM_{d} (Y)$ is isomorphic to the Hilbert scheme $\bH_{d} (V_5)$ of curves with Hilbert polynomial $dm+1$ (cf. \cite[Proposition 3.1]{Chu19}). Thus one can borrow the description of Hilbert scheme of rational curves in $V_5$.

\begin{proposition}[\cite{Fae05, FN89, Ili94, San14}]\label{radel}
$\bM_{1}\cong \PP^2$, $\bM_{2}\cong \PP^4$, $\bM_{3}\cong \Gr(2,5)$.
\end{proposition}

\begin{remark}
By taking an explicit choice of the hyperplanes $H_i$, the results for $d=1$ and $2$ of Proposition \ref{radel} has been reproved by the birational-geometric method (For the detail, see \cite[Section 7]{CHL18}).
\end{remark}

The universal sheaves over $\bM_d$ were explicitly presented in Proposition 2.20, Proposition 2.32,  and Proposition 2.46 of \cite{San14}. Let $i:\cC_d\hookrightarrow\bM_d\times V_5$ be the universal curve in $\bM_d\times V_5$. The free resolutions of $i_*\cO_{\cC_d}$ on $\bM_d\times V_5$ are
\begin{enumerate}
\item ($d=1$) $0\lr\cO_{\PP^2}(-3)\boxtimes\cS\lr\cO_{\PP^2}(-2)\boxtimes\cQ^*\lr \cO_{\PP^2\times V_5}\lr i_*\cO_{\cC_1}\lr0$,
\item ($d=2$) $0\lr\cO_{\PP^4}(-2)\boxtimes\cO_{V_5}(-1)\lr\cO_{\PP^4}(-1)\boxtimes\cS\lr \cO_{\PP^4\times V_5}\lr i_*\cO_{\cC_2}\lr0$,
\item ($d=3$) $0\lr\cS(-1)\boxtimes\cO_{V_5}(-1)\lr\cO_{\Gr(2,5)}(-1)\boxtimes\cQ(-1)\lr \cO_{\Gr(2,5)\times V_5}\lr i_*\cO_{\cC_3}\lr0$.
\end{enumerate}
Note that each of the moduli spaces is smooth, and we can check that the virtual class $[\bM_d]^{\text{vir}}$ is the Euler class of a K-theoretic obstruction bundle $\text{ob}_{\bM_d}$.
One can compute this K-theoretic obstruction bundle by the formula
\[
[\text{ob}_{\bM_d}] = [T_{\bM_d}] + [\mathbf{R}_{\pi_{\bM_{d}}}\mathbf{R}\cH om_{\bM_{d}\times V_5}(i_*\cO_{\cC_d},i_*\cO_{\cC_d})[1] \ ] - [\cO]
\] 
which appears in the proof of \cite[Proposition 2.13]{CT21}.
\begin{center}
\begin{tabular}{|l||M{7.7cm}|M{7.7cm}|}
\hline
$d$&$\qquad\mathbf{R}_{\pi_{\bM_{d}}}\mathbf{R}\cH om_{\bM_{d}\times V_5}(i_*\cO_{\cC_d},i_*\cO_{\cC_d})[1]$& $\qquad\qquad\qquad[\text{ob}_{\bM_d}]$\\
\hline
\hline
$1$&$-3[\cO_{\PP^2}]+3[\cO_{\PP^2}(1)]+5[\cO_{\PP^2}(2)]-5[\cO_{\PP^2}(3)]$&$[\cO_{\PP^2}]-5[\cO_{\PP^2}(2)]+5[\cO_{\PP^2}(3)]$\\
\hline
$2$&$-3[\cO_{\PP^4}]+10[\cO_{\PP^4}(1)]-7[\cO_{\PP^4}(2)]$&$[\cO_{\PP^4}]-5[\cO_{\PP^4}(1)]+7[\cO_{\PP^4}(2)]$\\
\hline
$3$&$-2[\cO_{\Gr(2,5)}]+10[\cO_{\Gr(2,5)}(1)]-7[\cS^*(1)]+5[\cS^*]-[\cS^*\otimes \cS]$&$[\cO_{\Gr(2,5)}]-10[\cO_{\Gr(2,5)}(1)]+7[\cS^*(1)]-5[\cS^*]+[\cS^*\otimes \cS]+[\cS^*\otimes \cQ]$.\\
\hline
\end{tabular}
\end{center}
Here, the bundles $\cS$ and $\cQ$ in the fourth row are the universal sub-bundle and quotient bundle of $\bM_3=\Gr(2,5)$.
By using the computer algebra system, \texttt{Macaulay2} (\cite{M2}), we have
\begin{proposition}\label{v5dt}
The invariants $\langle \tau_i( h_{2-i})\rangle_{d}$ are given by the numbers of the following table.
\begin{center}
\begin{tabular}{|l||M{2.2cm}|M{2.2cm}|}
\hline
$d$&$i=0$& $i=1$ \\
\hline
\hline
$1$&$5$&$\frac{25}{2}$\\
\hline
$2$&$35$&$\frac{35}{2}$\\
\hline
$3$&$490$&$-\frac{490}{2}$\\
\hline
\end{tabular}
\end{center}
\end{proposition}
\begin{proof}
Let us present the computation of the invariants for the degree $d=3$ case. The other cases are more simple and thus we omit it. The cohomology ring structure of $\Gr(2,5)$ is very well-known as follow. Let $m_i:=c_i(\cS)$ be the $i$-th Chern class of the universal subbundle $\cS$ of Grassmannian $\Gr(2,5)$. The cohomology ring of $\Gr(2,5)$ is given by (\cite[Theorem 5.26]{EH16})
\[
\rH^*(\Gr(2,5),\ZZ)=\ZZ[m_1,m_2]/\langle-m_1^5+4m_1^3m_2-3m_1m_2^2,\; m_1^4-3m_1^2m_2+m_2^2
\rangle.
\]
Note that the dual of the point class is $[\text{point}]^*=m_2^3$.
Then the Chern class of the obstruction bundle $\text{ob}_{\bM_d}$ is
\[
c(\text{ob}_{\bM_d})=1-11m_1+(48m_1^2+7m_2)+(-102m_1^3-56m_1m_2+451m_1^2m_2)-78m_2^2-490m_1m_2^2.
\]
Thus the virtual class of $\bM_3$ is $[\bM_3]^{\text{vir}}=-490m_1m_2^2$.

On the other hand, the insertion classes on $\rH^*(\bM_3\times V_5)$ are
\[\begin{split}
\tau_0(h_2)&=m_1^2h_2-m_2h_2-m_1h_3, \\
\tau_1(h_1)&= m_1^3h_1-\frac{3}{2}m_1m_2h_1-\frac{5}{2}m_1^2h_2+\frac{1}{2}m_1h_3
\end{split}
\]
From these one, we have
\[
\int_{[\bM_3]^{\text{vir}}} \tau_0(h_2)=490m_2^3h_3, \; \int_{[\bM_3]^{\text{vir}}} \tau_1(h_1)=-245m_2^3h_3.
\]
\end{proof}

\begin{remark}
From the description of the universal curve $\cC_1$ in \cite[Lemma 2.1, 2.2]{FN89}, one can easily check that the obstruction bundle is isomorphic to $\mathrm{ob}_{\bM_1}\cong \cO_{\bM_1}(5)$ and thus its the cohomology matches with our computation.
\end{remark}
\begin{remark}
The universal curve $\cC_2$ is a regular section of the vector bundle $\cO_{\PP^4}(1)\boxtimes \cS^*$ (\cite[Proposition 2.32]{San14}). Hence the Chern charcter is given by
\[
\text{ch}(i_*\cO_{\cC_2})=c_2(\cO_{\PP^4}(1)\boxtimes \cS^*)\cdot\text{td}(\cO_{\PP^4}(1)\boxtimes \cS^*)^{-1},
\]
and thus its cohomology class matches with our computation. In the following subsection, we find the fundamental class of 
$\cC_3$ by using Porteous’ formula.
\end{remark}
\subsubsection{The universal cubic curves $\cC_3$ via degeneracy loci}
Recall that the space $\bM_3$ is isomorphic to $\Gr(2,V)$ such that $\dim V=5$. In this subsection, we describe the universal family $\cC_3$ of cubic curves in a geometric way which confirms the calculation of previous subsection. Let us recall the isomorphism $\bM_3\cong \Gr(2,V)$. Consider the Schubert variety $$\sigma_{2,0}(l):=\{[l']\in \Gr(2,V)| l\cap l'\neq \emptyset\}$$ which is a degree $3$ and $4$-dimensional subvariety of $\Gr(2,V)$. By taking the hyperplane sections $H_1\cap H_2\cap H_3$ with this $\sigma_{2,0}(l)$, we obtain a twisted cubic curve
\[
C_l:=\sigma_{2,0}(l)\cap H_1\cap H_2\cap H_3\subset \Gr(2,5)\cap H_1\cap H_2\cap H_3=V_5,
\]
that is, $\pi_{V_5}(\pi_{\bM_3}^{-1}([l]))=C_l$. Conversely, for a point $p=[L]\in Y\subset \Gr(2,V)$, the inverse image $\pi_{V_5}^{-1}(p)$ consists of the twisted cubic curves $[C_l]$ such that $l\cap L\neq \emptyset$. This implies that $\pi_{\bM_3}(\pi_{V_5}^{-1}(p))=\sigma_{2,0}(L)\subset \Gr(2,V)=\bM_3$. Note that the Schubert variety $\sigma_{2,0}$ is a cone of rational normal scroll $\PP(\cO_{\PP^1}(1)^{\oplus3})\cong \PP^2\times\PP^1 $ in $\PP^6$.
Thus the universal cubic $\cC_3$ is a irreducible variety of dimension $7$. Also, it is well-known that the Schubert variety $\sigma_{2,0}$ can be defined by a degeneracy loci of vector bundles over Grassmannian. 
\begin{example}
Let $i:W\subset V$ be one-dimensional subvector space of $V$. For a line $l=\PP(W)\subset\PP(V)$,
let $$\phi:\cU\lr (V/W) \otimes \cO_{\Gr(2,V)}$$ be the canonical morphism induced by the injection $i$. Then one can check that the degeneracy locus $D_{\leq1}(\phi)\subset \Gr(2,V)$ of the map $\phi$ whose rank is $\leq 1$ has the support $\sigma_{2,0}(l)$. Also it has the expected dimension $\dim \Gr(2,V)-(2-1)\cdot(3-1)=4$. Thus, by Porteous' formula, the fundamental class of $D_{\leq1}(\phi)$ is given by
\[[D_{\leq1}(\phi)]= c_2([(V/W) \otimes \cO_{\Gr(2,V)}]-\cU).\]
\end{example}
In our case, by relativizing over $V_5$, we can find the fundamental form $[\cC_3]$ over $\bM_3\times V_5$.
\begin{proposition}
Let $\cS_{\bM_3}$ be the universal subbundle of $\bM_3$ and $\cQ_{V_5}$ be the restriction of the universal quotient bundle $V_5\subset \Gr(2,V)$. Then the fundamental class of the universal cubic curves is given by
\[
[\cC_3]=c_2(\cQ_{V_5}-\cS_{\bM_3})\in \rH^2(\bM_3\times Y).
\]
\end{proposition}
\begin{proof}
Since $V_5\subset \Gr(2,V)$, we have an universal sequence
\begin{equation}\label{eq1}
\ses{\cS_{V_5}}{V\otimes \cO_{V_5}}{\cQ_{V_5}}.
\end{equation}
Let us consider the relative Grassmannian bundle $\Gr(2, V\otimes \cO_{V_5})\lr V_5$ with the structure morphism $\pi_{V_5}:\Gr(2, V\otimes \cO_{V_5})\lr V_5$. Here we denote the same notation with the projection map because $\Gr(2, V\otimes \cO_{V_5})\cong \Gr(2,V)\times V_5$. From the universal sequence,
\[
\ses{\cS_{\bG}}{\pi_Y^*( V\otimes \cO_{V_5})}{\cQ_{\bG}}
\]
over $\bG:=\Gr(2, V\otimes \cO_{V_5})$ and the pull-back of the sequence \eqref{eq1}, we obtain a bundle morphism
\[
\phi_{\bG}:\cS_{\bG}\lr \pi_{V_5}^*\cQ_{V_5}
\]
over $\bG$. Note that $\cS_{\bG}=\pi_{\bM_3}^*\cS_{\bM_3}$ by its definition. The space $\cC_3$ is reduced because it is a generically reduced and Cohen–Macaulay space. Thus the degeneracy locus $D_{\leq1}(\phi_{\bG})$ of the map $\phi_{\bG}$ is $\cC_3$. By Porteous’ formula,
\[
[D_{\leq1}(\phi_{\bG})]=c_2(\cQ_{V_5}-\cS_{\bM_3}).
\]
\end{proof}
\begin{remark}
The Poincar\'e dual of the fundamental class of the universal cubic curves is
\[
[\cC_3]=m_1^2-m_2-m_1h_1+3h_2 \in \rH^2(\bM_3\times V_5),
\]
which matches our computation of Subsection \ref{casev5}.
\end{remark}

\subsubsection{The case $V_{22}$}\label{casev22}
Let us recall the definition of the variety $V_{22}$. Let $\cS$ and $\cQ$ be the universal bundles of $\Gr(3,7)$. Then $V_{22}$ is defined as a zero section of $\wedge^2(\cS^*)^{\oplus3}$. Alternatively, $V_{22}$ can be regraded as a subvariety of the net of quadrics $\bN(4;2,3)$.
\begin{itemize}
\item $\mathrm{Pic}_{\ZZ}(V_{22})\cong \ZZ\langle H\rangle$, $\mathrm{deg}(V_{22})=22$, $K_{V_{22}}=-H$.
\item The cohomology ring of $V_{22}$ over $\ZZ$ is isomorphic to
\[
\rH^{*}(V_{22},\ZZ)\cong\ZZ[h_1,h_2,h_3]/\langle h_1^2-22h_2,h_1^3-22h_3,h_1h_2-h_3,h_2^2\rangle
\]
where $\deg(h_i)=i, 1\leq i\leq 3$. Moreover, $h_1=c_1(\cO_{V_{22}}(1))$ and $h_i$ is the Poincar\'e dual of the linear space of dimension $3-i$ for $i=2,3$.
\item The Chern classes of tautological bundles on $V_{22}$ are
\begin{enumerate}
\item $c(\cS_{V_{22}})=1-h_1+10h_2-2h_3,$
\item $c(\cQ_{V_{22}})=1+h_1+12h_2+4h_3,$
\item $c({V_{22}})=1+h_1+24h_2+4h_3$
\end{enumerate}
where $\cS_{V_{22}}$ and $\cQ_{V_{22}}$ are the restriction of the universal bundles $\cS$ and $\cQ$ on $\Gr(3,7)$. Unless otherwise stated, we omit the subscription $V_{22}$ of the universal bundles.
\end{itemize}
By the same reason as in the case of $V_5$, the moduli space $\bM_{d} (V_{22})$ is isomorphic to the Hilbert scheme $\bH_{d} (V_{22})$ of curves with Hilbert polynomial $dm+1$ for $1\leq d \leq 3$. The later space $\bH_{d} (V_{22})$ has been studied by many authors.
\begin{proposition}[\cite{KPS18, Fae14, KS04}]
$\bM_{1}\cong Q$, $\bM_{2}\cong \PP^2$, $\bM_{3}\cong \PP^3$, where $Q$ is a singular planar quartic curve.
\end{proposition}
Let us compute the degree $d=1$ case by the result of Pirola (\cite{Pir85}).
By the Chern class computation, the virtual dimension of $\bM_1$ is $\text{virt. dim} \bM_1=1$ and the virtual fundamental class is given by the following. 
\begin{lemma}
\[[\bM_1]\virt = [\bM_1]\]
\end{lemma}
\begin{proof}
By the deformation invariance of DT invariants, we may assume that $V_{22}$ is not Mukai-Umemura 3-folds. Then $M_1 \cong Q \subset \PP^2$ is a regular embedding, hence the intrinsic normal cone $\cC_{\bM_1}$ of $\bM_1$ is given by the bundle stack
\[
\cC_{\bM_1} = [N_{Q/\PP^2} / T_{\PP^2}|_{Q}].
\]
Let $E \to L_{\bM_1}$ be the usual perfect obstruction theory. Then by \cite{BF97}, we have the closed embedding
\[
\cC_{\bM_1} \hra h^1/h^0(E\dual).
\]
Since they are both bundle stacks with dimension 0 (as Artin stacks), we have $\cC_{\bM_1} = h^1/h^0(E\dual)$. Hence, by the definition of the virtual cycle in \cite{BF97}, we have
\[
[\bM_1]\virt = 0^!_{{h^1/h^0(E\dual)}} [\cC_{\bM_1}] = 0^!_{{h^1/h^0(E\dual)}}[h^1/h^0(E\dual)] = [\bM_1]. 
\]
where the last equality comes from properties of Gysin pull-back via bundle stacks \cite{Kre99}.
\end{proof}

\begin{proposition}
The degree $d=1$ DT invariant and descendent invariant on $V_{22}$ are given by
\[
\langle \tau_0(h_2)\rangle_{1}=2\;\text{and}\; \langle \tau_1(h_1)\rangle_{1}=22
\]
\end{proposition}
\begin{proof}
Let $\cC_1$ be the universal curve over $\bM_1(=Q)$.
We compute the invariants by using the degeneracy loci method.
In \cite[Lemma 3.1]{AF06}, the authors describe how to obtain lines in $V_{22}$. We relativize their construction. Let $\cK$ be the vector bundle on $V_{22}$ with data $\rk(\cK)=5$, $c_1(\cK)=-2$, $c_2(\cK)=40$, $c_3(\cK)=-20$, and $\dim \Hom (\cK, \cS)=3$. Let $B^*=\Hom (\cK, \cS)$. Note that $Q\subset \PP(B^*)$. The universal curve $\cC_1\subset \PP(B^*)\times V_{22}$ is the degeneracy loci of the canonical homomorphism
\[
\Phi:\cS^*\lr \cK^*\boxtimes \cO_{\PP(B^*)}(1).
\]
In fact, the map $\Phi$ is the dual of the composition of the pull-back of the evaluation map $\text{ev}: \cK\otimes \Hom(\cK,\cS)\lr \cS$ on $V_{22}$ and the tautological map $\cO_{\PP(B^*)}(-1)\lr B\otimes \cO_{\PP(B^*)}$ on $\PP(B^*)$.

Let $\rH^\bullet(\PP(B^*),\ZZ)=\ZZ[m_1]/\langle m_1^3\rangle$ with $\deg(m_1)=1$.
By Proposition 3.14 in \cite{Pir85}, the Chern character of the structure sheaf $\cO_{\cC_1}$ over $\PP(B^*)\times V_{22}$ is given by
\[
\text{ch}(\cO_{\cC_1})=(2m_1^2h_1+4m_1h_2)+(-8m_1^2h_2+2m_1h_3)+(\text{the terms of higher degree}).
\]
Thus, by the Grothendieck-Riemann-Roch theorem, the invariants are
\[
\int_{[\bM_{1}]} \tau_0(h_2) =2m_1^2h_3 = 2[pt],\;\int_{[\bM_{1}]} \tau_1(h_1) =22m_1^2h_3 = 22[pt].
\]
\end{proof}
\begin{remark}
By Porteous’ formula, the dual class of the fundamental class $[\cC_1]$ is
\[
[\cC_1]=2m_1^2h_1+4m_1h_2.
\]
The intersection number of $[\cC_1]$ with the line class $h_2$ in $V_{22}$ is $[\cC_1]\cdot h_2=2m_1^2h_3$. This matches with the fact that the degree of the surface $S$ sweeping out by lines in $V_{22}$ is $\deg(S)=2$ (\cite[Section 3]{Ame98}).
\end{remark}
For the degree $d=2$ and $3$ cases, the universal curves over $\bM_d$ have been studied in \cite[Lemma 4.1]{Fae14} and \cite[Theorem 2.4]{KS04}. Let $i:\cC_d\hookrightarrow\bM_d\times V_{22}$ be the universal curve in $\bM_d\times V_{22}$. The free resolutions of $i_*\cO_{\cC_d}$ on $\bM_d\times V_{22}$ are
\begin{enumerate}
\item ($d=2$) $0\lr \cS\boxtimes\cO_{\PP^2}(-4)\lr \cQ^*\boxtimes\cO_{\PP^2}(-3)\lr \cO_{V_{22}\times \PP^2} \lr i_*\cO_{\cC_2}\lr 0$,
\item ($d=3$) $0\lr \cE\boxtimes \cO_{\PP^3}(-3)\lr \cS\boxtimes \cO_{\PP^3}(-2) \lr \cO_{V_{22}\times \PP^3}\lr i_*\cO_{\cC_3}\lr 0$,
\end{enumerate}
where $\text{rk}(\cE)=2$, $c_1(\cE)=-1$, $c_2(\cE)=7$.
By the same method for the case $V_5$, one can find the (virtual) fundamental class $[\bM_d]$.
\begin{center}
\begin{tabular}{|l||M{7.7cm}|M{7.5cm}|}
\hline
$d$&$\qquad\mathbf{R}_{\pi_{\bM_{d}}}\mathbf{R}\cH om_{\bM_{d}\times V_{22}}(i_*\cO_{\cC_d},i_*\cO_{\cC_d})[1]$& $\qquad\qquad\qquad[\text{ob}_{\bM_d}]$\\
\hline
\hline
$2$&$-3[\cO_{\PP^2}]+7[\cO_{\PP^2}(3)]-7[\cO_{\PP^2}(4)]+3[\cO_{\PP^2}(1)]$&$[\cO_{\PP^2}]-7[\cO_{\PP^2}(3)]+7[\cO_{\PP^2}(4)]$\\
\hline
$3$&$-3[\cO_{\PP^3}]+7[\cO_{\PP^3}(2)]-8[\cO_{\PP^3}(3)]+4[\cO_{\PP^3}(1)]$&$[\cO_{\PP^3}]-7[\cO_{\PP^2}(2)]+8[\cO_{\PP^2}(3)]$\\
\hline
\end{tabular}
\end{center}
Therefore we have
\begin{proposition}\label{v22dt}
The invariants $\langle \tau_i(h_{2-i})\rangle_{d}$ are given by the numbers of the following table.
\begin{center}
\begin{tabular}{|l||M{2.2cm}|M{2.2cm}|}
\hline
$d$&$i=0$& $i=1$ \\
\hline
\hline
$1$&$2$&$22$\\
\hline
$2$&$7$&$28$\\
\hline
$3$&$28$&$28$\\
\hline
\end{tabular}
\end{center}
\end{proposition}
By combining Proposition \ref{gwfano}, Proposition \ref{v5dt} and Proposition \ref{v22dt}, we have

\begin{theorem}\label{main1}
Conjecture \ref{conj:genus0} (which is equivalent to Conjecture \ref{gw-dtconj}) is true when $X=V_5$ and $V_{22}$ up to the degree $3$.
\end{theorem}

\section{Proof of Conjecture \ref{conj:genus1}}\label{meeting}
Let us recall the definition of meeting invariants $m_{\beta_1,\beta_2}\in \ZZ$ for $\beta_1,\beta_2\in \rH_2(X,\ZZ)$ (\cite[Section 0.3]{KP08}). It is given by the following rules:
\begin{enumerate}
\item $m_{\beta_1,\beta_2}=m_{\beta_2,\beta_1}$.
\item If either $\deg(\beta_1)\leq 0$ or $\deg(\beta_2)\leq 0$, then $m_{\beta_1,\beta_2}=0$.

Let $\{S_1,\cdots, S_k\}$ be the basis of the torsion free part of $\rH^4(X,\ZZ)$. Let $(g^{ij})$ be the inverse matrix of the intersection matrix $(g_{ij})$, $g_{ij}=\langle S_i, S_j\rangle$.
\item If $\beta_1\neq \beta_2$,
\[
m_{\beta_1,\beta_2}=\sum_{i,j}n_{0,\beta_1}(S_i) g^{ij} n_{0,\beta_2}(S_j)+m_{\beta_1,\beta_2-\beta_1}+m_{\beta_1-\beta_2,\beta_2}.
\]
\item If $\beta_1= \beta_2=\beta$, then
\[
m_{\beta,\beta}=n_{0,\beta}(c_2(X))+\sum_{i,j}n_{0,\beta}(S_i) g^{ij} n_{0,\beta}(S_j)-\sum_{\beta_1+\beta_2=\beta}m_{\beta_1,\beta_2}.
\]
\end{enumerate}
We recall Conjecture \ref{conj:genus1} here.
For each $\gamma\in \rH^2(X,\ZZ)$,
\[
 \langle \tau_1(\gamma) \rangle_{\beta} =\frac{n_{0,\beta}(\gamma^2)}{2(\gamma\cdot \beta)}-\sum_{\beta_1+\beta_2=\beta} \frac{(\gamma\cdot \beta_1)(\gamma\cdot \beta_2)}{4(\gamma\cdot \beta)} m_{\beta_1,\beta_2} 
-\sum_{k\geq1, k|\beta} \frac{(\gamma\cdot \beta)}{k}n_{1,\beta/k}
\]
\begin{remark}
It is believed that the invariants $n_{1,d}$ come from a space of elliptic curves in $Y$ (cf. \cite[Section 3 and 5]{KP08}). The defining equations of $V_{5}$ and $V_{22}$ can be generated by quadric equations and they does not contain planes. Thus the space of elliptic curves of degree $d\leq 3$ should be empty. This implies that the invariants are $n_{1,d}=0$.
\end{remark}
In this section, we will prove the following.

\begin{theorem}\label{pfmeeting}
Conjecture \ref{conj:genus1} holds for $Y=V_{5}$ and $V_{22}$ when we assume that $n_{1,d}=0$ for $1\leq d\leq 3$.
\end{theorem}
The remaining two subsections is devoted to the proof of Theorem \ref{pfmeeting}.
\subsection{The case $V_5$}
Let $X = |K_Y|$, $\bar{X} = \PP(K_Y \oplus \cO_Y)$ and $\pi : \bar{X} \to Y$ be the canonical projection map. By construction, we have
\[
\rH^4(\bar{X},\ZZ) = \langle T_1, T_2  \rangle,
\]
where $T_1 = \text{PD}( H \cap Y )$, $T_2 = \pi^* ( \text{PD}(L_1 ))$ for linear spaces $L_i$ of dimension $i$. Since the normal bundle of $H\cap Y$ in $X$ is $\cN_{H\cap Y, X}\cong\pi^* (\cO_Y(1) \oplus \cO_Y(-2) )$, $T_1 \cdot T_1 = -2H^3 \cap Y = -10$. Hence the intersection matrix is given by
\begin{center}
$(g_{ij})=$
\begin{tabular}{|M{2.2cm}|M{2.2cm}|M{2.2cm}|}
\hline
&$T_1$&$T_2$\\
\hline
$T_1$&$-10$&$1$\\
\hline
$T_2$&$1$&$0$\\
\hline
\end{tabular}\;,
$g^{ij} = (g_{ij})^{-1} = \left(
\begin{matrix}
0 & 1 \\
1 & 10 \\
\end{matrix}
\right).
$
\end{center}
From $\cT_X = \pi^*( \cT_Y \oplus \cO_Y(-2) )$, we have $c(X) = \pi^*(c(Y)\cdot (1 - 2h_1))$. Hence $c_2(X)= -8T_2$. Also we have $T_1|_{Y} = (T_1 \cdot T_1)T_2|_{Y} = -10 T_2|_Y $.
Using the formula of meeting invariants, we have
\begin{align*}
m_{1,1} & = n_{0,1}(-8T_2) + 2 n_{0,1}(T_1)n_{0,1}(T_2) + 10 n_{0,1}(T_2)^2 \\
& = 40 - 20 \times 5^2 + 10 \times 5^2 = -210, \\
m_{1,2} & = n_{0,1}(T_1)n_{0,2}(T_2) + n_{0,1}(T_2)n_{0,2}(T_1) + 10 n_{0,1}(T_2)n_{0,2}(T_2) + m_{1,1} \\
& = -10 n_{0,1}(T_2) n_{0,2} (T_2) - 210 = -1960.
\end{align*}
Motivated from the fact that the spaces of genus one curves on $Y$ with degree $d=1,2,3$ are empty sets, let us assume that $n_{1,1} = n_{1,2} = n_{1,3} = 0$. Then Conjecture \ref{conj:genus1} has been written as
\[
\langle \tau_1( H ) \rangle_d = \frac{n_{0,d}(H^2)}{2d} - \sum_{d_1 + d_2 = d} \frac{d_1\cdot d_2}{4d}m_{d_1,d_2}.
\]
Note that we need to choose some suitable orientation of the moduli space. Here we use the orientation in \cite[(0.7)]{CT21}. By a direction calculation, one can  check the identity as follows.
\[\begin{split}
-\langle \tau_1(H) \rangle_1 &= \frac{5 n_{0,1}(T_2)}{2},\;-\langle \tau_1(H) \rangle_2 = \frac{5n_{0,2}(T_2)}{4} - \frac{1}{8}m_{1,1},\\
-\langle \tau_1(H) \rangle_3 &= \frac{5n_{0,3}(T_2)}{6} - \frac{1}{3}m_{1,2}
\end{split}\]
Therefore the suitable choice of sign for the conjecture should be $-1$.
\subsection{The case $V_{22}$}
Let $Y = V_{22}$, $X = |K_{V_{22}}|$, and $\bar{X} = \PP(K_Y \oplus \cO_Y)$ and $\pi : \bar{X} \to Y$ be the canonical projection map.
Let $T_1$ and $T_2$ be generators of the cohomology group $\rH^2(X,\ZZ)$, defined by
\[
T_1 := PD(H\cap Y), \ T_2:=\pi^*(PD(L_1))
\]
where $L_1$ is a class of line, equal to $(H^2/22) \cap Y$. Then the intersection matrix is computed by:
\begin{center}
$(g_{ij})=$
\begin{tabular}{|M{2.2cm}|M{2.2cm}|M{2.2cm}|}
\hline
&$T_1$&$T_2$\\
\hline
$T_1$&$-22$&$1$\\
\hline
$T_2$&$1$&$0$\\
\hline
\end{tabular}
\;,\; $g^{ij} = (g_{ij})^{-1} = \left(
\begin{matrix}
0 & 1 \\
1 & 22 \\
\end{matrix}
\right)
$
\end{center}
Note that we can check $c_2(X) = 2T_2$ by direct calculation. Also, by  the same manner as did in the case $Y = Y_5$, the meeting invariants are
\[
m_{1,1} = -84, \ m_{1,2} = 224.
\]
Under the assumption $n_{1,1} = n_{1,2} = n_{1,3} = 0$, we have the identities:
\[\begin{split}
\langle \tau_1(H) \rangle_1 &= \frac{22n_{0,1}(T_2)}{2},\;-\langle \tau_1(H) \rangle_2 = \frac{22 n_{0,2}(T_2)}{4} - \frac{1}{8}m_{1,1},\\
\langle \tau_1(H) \rangle_3 &= \frac{22 n_{0,3}(T_2)}{6} - \frac{1}{3}m_{1,2},
\end{split}\]
which confirms Theorem \ref{pfmeeting}.

\bibliographystyle{alpha}
\newcommand{\etalchar}[1]{$^{#1}$}

\end{document}